\documentclass[twoside]{article}
\usepackage{amssymb}
\usepackage{indentfirst}
\usepackage{dsfont}
\usepackage{amsmath,amsthm,amsfonts,amssymb,graphicx}
\usepackage[colorlinks,linkcolor=blue,anchorcolor=blue,citecolor=green]{hyperref}
\usepackage{bibentry}
\allowdisplaybreaks

\newtheorem{theorem1}{Theorem}[section]
\newtheorem{proposition}{Proposition}[section]
\newtheorem{lemma}{Lemma}[section]
\numberwithin{equation}{section}
\begin{document}
\baselineskip 14pt \noindent \thispagestyle{empty}
\bibliographystyle{ieeetr}

\markboth{\centerline{\rm E. Bi \; \& \; Z. Tu}}{\centerline{\rm
Cartan-Hartogs domains}}

\begin{center} \Large\bf{Remarks on the canonical metrics on the Cartan-Hartogs domains}
\end{center}

\begin{center}
\noindent\text{Enchao Bi$^1$ $^{*}$} \; \&  \; Zhenhan Tu$^2$ \\
\vskip 4pt

\noindent\small {${}^1$School of Mathematics and Statistics, Qingdao
University, Qingdao, Shandong 266071, P.R. China} \\

\noindent\small {${}^2$School of Mathematics and Statistics, Wuhan
University, Wuhan, Hubei 430072, P.R. China} \\

\noindent\text{Email: bienchao@whu.edu.cn (E. Bi),\;
zhhtu.math@whu.edu.cn (Z. Tu)}
\renewcommand{\thefootnote}{{}}
\footnote{\hskip -16pt {$^{*}$Corresponding author. \\ } }
\end{center}

\begin{center}
\begin{minipage}{13cm}
{\bf Abstract}\hskip 10pt The Cartan-Hartogs domains are defined as
a class of Hartogs type domains over irreducible bounded symmetric
domains. For a Cartan-Hartogs domain $\Omega^{B}(\mu)$ endowed with
the natural K\"{a}hler metric $g(\mu),$ Zedda conjectured that the
coefficient $a_2$ of the Rawnsley's $\varepsilon$-function expansion
for the Cartan-Hartogs domain $(\Omega^{B}(\mu), g(\mu))$ is
constant on $\Omega^{B}(\mu)$ if and only if $(\Omega^{B}(\mu),
g(\mu))$ is biholomorphically isometric to the complex hyperbolic
space. In this paper, following  Zedda's argument, we give a
geometric proof of the Zedda's conjecture
by computing the curvature tensors of the Cartan-Hartogs domain $(\Omega^{B}(\mu), g(\mu))$. \\

{\bf Keywords}\hskip 10pt  Bounnded symmetric domains
\textperiodcentered \, Cartan-Hartogs domains \textperiodcentered \,
Curvature tensors \textperiodcentered \, K\"{a}hler metrics \\

  \textperiodcentered \, 32M15  \textperiodcentered \, 32Q15
\end{minipage}
\end{center}

\section{Introduction}

The expansion of the Bergman kernel has received a lot of attention
recently, due to the influential work of Donaldson, see e.g.
\cite{Donaldson}, about the existence and uniqueness of constant
scalar curvature K\"{a}hler metrics (cscK metrics). Donaldson used
the asymptotics of the Bergman kernel proved by Catlin \cite{Cat}
and Zelditch \cite{Zeld} and the calculation of Lu \cite{Lu} of the
first coefficient in the expansion to give conditions for the
existence of cscK metrics.

Assume that $D$ is a bounded domain in $\mathbb{C}^n$ and $\varphi$
is a strictly plurisubharmonic function on $D$. Let $g$ be a
K\"{a}hler  metric  on $D$ associated to the K\"{a}hler form
$\omega=\frac{\sqrt{-1}}{2}\partial\overline{\partial}\varphi$. For
$\alpha>0$, let $\mathcal{H}_{\alpha}$ be the weighted Hilbert space
of square integrable holomorphic functions on $(D, g)$ with the
weight $\exp\{-\alpha \varphi\}$, that is,
$$\mathcal{H}_{\alpha}:=\left\{ f\in \textmd{Hol}(D)\;\left |\; \int_{D}\right.|f|^2\exp\{-\alpha \varphi\}\frac{\omega^n}{n!}<+\infty\right\},$$
where $\textmd{Hol}(D)$ denotes the space of holomorphic functions
on $D$. Let $K_{\alpha}$ be the Bergman kernel (namely, the
reproducing kernel) of $\mathcal{H}_{\alpha}$ if
$\mathcal{H}_{\alpha}\neq \{0\}$.  The Rawnsley's
$\varepsilon$-function on $D$ (see Cahen-Gutt-Rawnsley \cite{CGR}
and Rawnsley \cite{R}) associated to the metric $g$  is defined by
\begin{equation*}
 \varepsilon_{\alpha}(z):=\exp\{-\alpha \varphi(z)\}K_{\alpha}(z,\overline{z}),\;\; z\in D.
\end{equation*}
Note the Rawnsley's $\varepsilon$-function depends only on the
metric $g$ and not on the choice of the K\"{a}hler potential
$\varphi$ (which is defined up to an addition with the real part of
a holomorphic function on $D$).
The asymptotics of the Rawnsley's
$\varepsilon$-function $\varepsilon_{\alpha}$ was expressed in terms
of the parameter $\alpha$ for compact manifolds by Catlin \cite{Cat}
and Zelditch \cite{Zeld} (for $\alpha\in \mathbb{N}$) and for
non-compact manifolds by Ma-Marinescu \cite{MM07,MM08}. In some
particular cases it was also proved by Engli\v{s} \cite{End2,Eng1}.

The Cartan-Hartogs domains are defined as a class of Hartogs type
domains over irreducible bounded symmetric domains. Let $\Omega$ be
an irreducible bounded symmetric domain in $\mathbb{C}^d$ of genus
$\gamma$. The generic norm of $\Omega$ is defined by
$$ N(z,
\xi) := (V(\Omega ) K(z, \xi))^{-1/\gamma},$$ where $V(\Omega)$ is
the total volume of $\Omega$ with respect to the Euclidean measure
of $\mathbb{C}^d$ and $K(z, {\xi})$ is its Bergman kernel. Thus $0<
N_{\Omega}(z,\bar{z})\leq 1$ for all $z\in \Omega$ and
$N_{\Omega}(0,0)=1$. For an irreducible bounded symmetric domain
$\Omega$ in $\mathbb{C}^d$ and a positive real number $\mu$, the
Cartan-Hartogs domain $\Omega^{B}(\mu)$ is defined by
\begin{equation*}
  \Omega^{B}(\mu):=\left\{(z,w)\in  \Omega\times \mathbb{C}:\; |w|^2<N(z,{z})^{\mu}
  \right\}.
\end{equation*}
For the Cartan-Hartogs domain $\Omega^{B}(\mu)$, define
\begin{equation*}
 \Phi(z,w):=-\log(N(z,{z})^{\mu}-|w|^2).
\end{equation*}
The K\"{a}hler form $\omega(\mu)$ on  $\Omega^{B}(\mu)$ is defined
by
\begin{equation*}
  \omega(\mu):=\frac{\sqrt{-1}}{2}\partial
\overline{\partial}\Phi.
\end{equation*}
We endow the Cartan-Hartogs
domain $\Omega^{B}(\mu)$ with the K\"{a}hler metric $g(\mu)$ associated to the K\"{a}hler form $w(\mu)$.
For the Cartan-Hartogs domain $(\Omega^{B}(\mu), g(\mu)),$ the
Rawnsley's $\varepsilon$-function admits the following finite expansion (e.g., see Th. 3.1 in Feng-Tu \cite{F-Tu}):
\begin{equation}
  \varepsilon_{\alpha}(z,w)=\sum_{j=0}^{d+1}a_j(z,w)\alpha^{d+1-j}, \;\; (z,w)\in
  \Omega^{B}(\mu).
\end{equation}
By Th. 1.1 of Lu \cite{Lu},  Th. 4.1.2 and Th. 6.1.1 of Ma-Marinescu
\cite{MM07}, Th. 3.11 of Ma-Marinescu \cite{MM08} and Th. 0.1 of
Ma-Marinescu \cite{MM12}, see also Th. 3.3 of Xu \cite{X}, we have

\begin{equation}\label{eauq3}
\left\{
\begin{aligned}
&a_{0}=1,\\
&a_{1}=\frac{1}{2}k_{g(\mu)},\\
&a_{2}=\frac{1}{3}\Delta k_{g(\mu)}+\frac{1}{24}{\vert R_{g(\mu)}\vert}^{2}-\frac{1}{6}{\vert Ric_{g(\mu)}\vert}^{2}+\frac{1}{8}k_{g(\mu)}^{2},\\
\end{aligned}
\right.
\end{equation}
where $k_{g(\mu)}$, $\triangle$, $R_{g(\mu)}$ and $Ric_{g(\mu)}$  denote  the scalar
curvature, the Laplace, the curvature tensor and the Ricci curvature
associated to the metric $g(\mu)$ on the Cartan-Hartogs domain
$\Omega^{B}(\mu)$, respectively.

Let $ {{B}}^{d}$ be the unit ball in $\mathbb{C}^d$  and  let the
metric $g_{hyp}$ on  $ {{B}}^{d}$  be given by
$$
ds^2=-\sum_{i,j=1}^{d}\frac{\partial^2\ln(1-\|z\|^2)}{\partial z_i
\partial \overline{z_j}}dz_i\otimes d\overline{z_j}.$$
We denote by $({B}^d, g_{hyp})$ the complex hyperbolic space. Note that
$g_{hyp}=\frac{1}{d+1}\, g_{B}$ on $ {{B}}^{d}$ for the Bergman
metric $g_{B}$ of $ {{B}}^{d}$. When $\Omega$ is the unit ball in
$\mathbb{C}^d$ and $\mu=1$, then the Cartan-Hartogs domain
$(\Omega^{B}(\mu), g(\mu))$ is the complex hyperbolic space in
$\mathbb{C}^{d+1}.$ With the exception of the complex hyperbolic
space which is obviously homogeneous, each Cartan-Hartogs domain
$(\Omega^{B}(\mu), g(\mu))$ is a noncompact, nonhomogeneous,
complete K\"{a}hler manifold. Further, for some particular value
$\mu_0$ of $\mu$, $g(\mu_0)$ is a K\"{a}hler-Einstein metric (see
Yin-Wang \cite{YW}).

Recently, Loi-Zedda \cite{LZ} and Zedda \cite{Zedda} studied the
canonical metrics on the Cartan-Hartogs domains. By calculating the
scalar curvature $k_{g(\mu)}$, the Laplace $\Delta k_{g(\mu)}$ of
$k_{g(\mu)}$, the norm $|R_{g(\mu)}|^2$ of the curvature tensor
$R_{g(\mu)}$ and the norm $|Ric_{g(\mu)}|^2$ of the Ricci curvature
$Ric_{g(\mu)}$ of a Cartan-Hartogs domain $(\Omega^{B^{d_0}}(\mu),
g(\mu))$, Zedda \cite{Zedda} has proved that if the coefficient
$a_2$ of the Rawnsley's $\varepsilon$-function expansion for the
Cartan-Hartogs domain $(\Omega^{B}(\mu), g(\mu))$ is constant on
$\Omega^{B}(\mu)$, then $(\Omega^{B}(\mu),g(\mu))$ is
K\"{a}hler-Einstein. Moreover, Zedda \cite{Zedda} conjectured that
the coefficient $a_2$ of the Rawnsley's $\varepsilon$-function
expansion for the Cartan-Hartogs domain $(\Omega^{B}(\mu), g(\mu))$
is constant on $\Omega^{B}(\mu)$ if and only if $(\Omega^{B}(\mu),
g(\mu))$ is biholomorphically isometric to the complex hyperbolic
space.

In 2014, Feng-Tu \cite{F-Tu} proved this conjecture by giving the
explicit expression of the the Rawnsley's $\varepsilon$-function
expansion for the Cartan-Hartogs domain $(\Omega^{B}(\mu), g(\mu))$.
The methods in Feng-Tu \cite{F-Tu} are very different from the
argument in Zedda \cite{Zedda}. In this paper, following the
framework of Zedda \cite{Zedda},  we give a geometric proof of the
Zedda's conjecture by computing the curvature tensors of the
Cartan-Hartogs domain $(\Omega^{B}(\mu), g(\mu))$. We will prove the
following result:

\begin{theorem1}
Let $(\Omega^{B}(\mu),g(\mu))$ be a Cartan-Hartogs domain endowed
with the canonical metric $g(\mu)$. Then  the coefficient $a_2$ of
the Rawnsley's $\varepsilon$-function expansion for the
Cartan-Hartogs domain $(\Omega^{B}(\mu), g(\mu))$ is constant on
$\Omega^{B}(\mu)$ if and only if $(\Omega^{B}(\mu), g(\mu))$ is
biholomorphically isometric to the complex hyperbolic space.
\end{theorem1}

Let $\Omega$ be the irreducible bounded symmetric domain endowed
with its Bergman metric $g_{B}$ and let $R_{g_{B}}$ denote the
curvature tensor associated to $(\Omega, g_{B})$. When the
coefficient $a_2$ of the Rawnsley's $\varepsilon$-function expansion
for the Cartan-Hartogs domain $(\Omega^{B}(\mu), g(\mu))$ is
constant on $\Omega^{B}(\mu)$, we will use the curvature tensor
$R_{g_{B}}$ on $\Omega$ to determine $\Omega^{B}(\mu)$, which means,
in this case, $(\Omega^{B}(\mu), g(\mu))$ must be biholomorphically
isometric to the complex hyperbolic space.

As general references for this paper, see Feng-Tu \cite{F-Tu} and
Zedda \cite{Zedda}. For the sake of simplicity, simliar to Zedda
\cite{Zedda}, the Cartan-Hatogs domains in this paper will be
restricted to the Hartogs type domains over the irreducible bounded
symmetric domains  only with one-dimensional fibers.

\section{Preliminaries}

Let $\Omega$ be an irreducible bounded symmetric domain of
$\mathbb{C}^{d}$ of genus $\gamma$ and let $N(z,z)$ denote the
generic norm of $\Omega$. Define
\begin{equation}\label{equa666}
g^{\Omega(\mu)}:=\frac{\mu}{\gamma}g_{B},
\end{equation}
where $g_{B}$ is the Bergman metric of $\Omega$. Then, from Zedda
\cite{Zedda}, we have the following results.

\begin{lemma}[Zedda \cite{Zedda}, Lemma 4]
The scalar curvature $k_{g(\mu)}$ of the Cartan-Hartogs domain
$(\Omega^{B}(\mu),g(\mu))$ is given by
\begin{equation}\label{equa6}
k_{g(\mu)}=\frac{d(\mu(d+1)-\gamma)}{\mu}\frac{N^{\mu}-{\vert w\vert}^{2}}{N^{\mu}}-(d+2)(d+1).
\end{equation}
\end{lemma}

\begin{lemma}[Zedda \cite{Zedda}, Lemma 8]
The norm with respect to $g(\mu)$ of the curvature tensor $R_{g(\mu)}$ of the
Cartan-Hartogs domain $(\Omega^{B}(\mu),g(\mu))$ when evaluated at
any point $(0,w)\in \Omega^{B}(\mu)\subseteq \Omega \times
\mathbb{C}$ is given by
\begin{equation}\label{equ7}
[{\vert R_{g(\mu)}\vert}^{2}]_{z=0}=(1-{\vert w\vert}^{2})^{2}{\vert
R_{g^{\Omega(\mu)}}\vert}^{2}-4{\vert w\vert}^{2} (1-{\vert
w\vert}^{2})k_{g^{\Omega(\mu)}}+2d(d+1){\vert w\vert}^{4}+4(d+1),
\end{equation}
where  $k_{g^{\Omega(\mu)}}$ is the scalar curvature of
$(\Omega,g^{\Omega(\mu)})$ and ${\vert R_{g^{\Omega(\mu)}}\vert}$ is
the norm with respect to $g^{\Omega(\mu)}$ of the curvature tensor
$R_{g^{\Omega(\mu)}}$ of   $(\Omega,g^{\Omega(\mu)})$.
\end{lemma}

\begin{lemma}[Zedda \cite{Zedda}, (39) and (40)]
For the Cartan-Hartogs domain $\Omega^{B}(\mu)$ endowed with the
K\"{a}hler metric $g(\mu)$, we have the following identities
\begin{equation}\label{equ8}
\begin{aligned}
{[\Delta k_{g(\mu)}]}_{z=0}=&-\frac{d(\mu(d+1)-\gamma)}{\mu}(1-{\vert w\vert}^{2})((d-1){\vert w\vert}^{2}+1),\\
\end{aligned}
\end{equation}
\begin{equation}\label{equ9}
\begin{aligned}
{[{\vert Ric_{g(\mu)}\vert}^{2}]}_{z=0}=&d(\frac{d(\mu(d+1)-\gamma)}{\mu})^{2}(1-{\vert w\vert}^{2})^{2}+\\
-&2d(d+2)\frac{d(\mu(d+1)-\gamma)}{\mu}(1-{\vert w\vert}^{2})+(d+1)(d+2)^{2}.\\
\end{aligned}
\end{equation}
\end{lemma}

\section{The proof of the main theorem}
In this section, we will give the proof of the main theorem.
Firstly, By (\ref{eauq3}), we have
\begin{equation} \label{30}
a_{2}(z,w)=\frac{1}{3}\Delta k_{g(\mu)}+\frac{1}{24}{\vert
R_{g(\mu)}\vert}^{2}-\frac{1}{6} {\vert
Ric_{g(\mu)}\vert}^{2}+\frac{1}{8}k_{g(\mu)}^{2}.
\end{equation}
For convenience, denote
\begin{equation} \label{300}
c:=\frac{\mu(d+1)-\gamma}{\mu}.
\end{equation}

If $a_{2}(z,w)$ is a constant, then $a_{2}(0,w)$ is also a constant
on $|w|<1$. Then by $(\ref{equa6})$, $(\ref{equ7})$, $(\ref{equ8})$
and $(\ref{equ9})$, after a straightforward computation, we
 have the following result.

 \begin{proposition} Let $\Omega$ be an irreducible bounded symmetric domain of
$\mathbb{C}^{d}$ of genus $\gamma$. Assume that
$(\Omega^{B}(\mu),g(\mu))$ is a Cartan-Hartogs domain endowed with
the canonical metric $g(\mu)$. If the coefficient $a_2$ of the
Rawnsley's $\varepsilon$-function expansion for the Cartan-Hartogs
domain $(\Omega^{B}(\mu), g(\mu))$ is constant on $\Omega^{B}(\mu)$,
then we have
\begin{equation}\label{equa17}
 \mu=\frac{\gamma}{d+1}, \;\;\; [{\vert
 R_{g_{B}}\vert}^{2}]_{z=0}=\frac{2d}{d+1},
\end{equation}
where $g_{B}$ is the Bergman metric of $\Omega$.
\end{proposition}

\begin{proof}
Firstly, by $(\ref{equ8})$, we have
\begin{equation} \label{31}
\begin{aligned}
\frac{1}{3}[\Delta k_{g(\mu)}]_{z=0}=&\frac{1}{3}d\cdot c(d-1){\vert w\vert}^{4}-\frac{1}{3}d\cdot c(d-2){\vert w\vert}^{2}-\frac{1}{3}dc.\\
\end{aligned}
\end{equation}
Since ${\vert R_{g^{\Omega(\mu)}}\vert}^{2}={\vert
R_{g_{B}}\vert}^{2}\frac{\gamma^2}{\mu^2}$ by the definition of
$g^{\Omega(\mu)}$ (see \eqref{equa666}), from $(\ref{equ7})$, we get
\begin{equation}\label{32}
\begin{aligned}
\frac{1}{24}[{\vert R_{g(\mu)}\vert}^{2}]_{z=0}=&\frac{1}{24}[\frac{\gamma^{2}}{\mu^{2}}{\vert R_{g_{B}}\vert}^{2}-4d\frac{\gamma}{\mu}+2d(d+1)]{\vert w\vert}^{4}
+\frac{1}{24}[-2\frac{\gamma^{2}}{\mu^{2}}{\vert R_{g_{B}}\vert}^{2}\\
+&4d\frac{\gamma}{\mu}]{\vert w\vert}^{2}
+\frac{1}{24}[\frac{\gamma^{2}}{\mu^{2}}{\vert R_{g_{B}}\vert}^{2}+4(d+1)].
\end{aligned}
\end{equation}
Similarly, from $(\ref{equ9})$, one rewrite the $-\frac{1}{6}[{\vert
Ric_{g(\mu)}\vert}^{2}]_{z=0}$ in ${\vert w\vert}^{4}$ and ${\vert
w\vert}^{2}$ as follows
\begin{equation}\label{33}
\begin{aligned}
-\frac{1}{6}[{\vert Ric_{g(\mu)}\vert}^{2}]_{z=0}=&-\frac{1}{6}d\cdot c^{2}{\vert w\vert}^{4}-\frac{1}{6}[2d(d+2)c-2dc^{2}]{\vert w\vert}^{2}+\\
-&\frac{1}{6}[dc^{2}+(d+1)(d+2)^{2}-2d(d+2)c].
\end{aligned}
\end{equation}
Lastly, from $(\ref{equa6})$, we have
\begin{equation} \label{34}
\begin{aligned}
\frac{1}{8}[k_{g(\mu)}^{2}]_{z=0}=&\frac{1}{8}d^{2}c^{2}{\vert w\vert}^{4}+\frac{1}{8}[2d(d+1)(d+2)c-2d^{2}c^{2}]{\vert w\vert}^{2}
+\frac{1}{8}[d^{2}c^{2}\\
+&(d+2)(d+1)^{2}].
\end{aligned}
\end{equation}

Combining \eqref{30},  \eqref{31}, \eqref{32}, \eqref{33} and
\eqref{34}, we have
$$[a_{2}(z,w)]_{z=0}=c_{0}{\vert w\vert}^{4}+c_{1}{\vert w\vert}^{2}+c_{2},$$
where
\begin{equation}
\begin{aligned}
c_{0}:=&\frac{1}{3}dc(d-1)+\frac{1}{24}\frac{\gamma^{2}}{\mu^{2}}{\vert R_{g_{B}}\vert}^{2}+\frac{1}{12}d(d+1)-\frac{1}{6}d\frac{\gamma}{\mu}\\
-&\frac{1}{6}dc^{2}+\frac{1}{8}d^{2}c^{2},\\
\end{aligned}
\end{equation}
\begin{equation}
\begin{aligned}
c_{1}:=&-\frac{1}{3}dc(d-2)-\frac{1}{12}\frac{\gamma^{2}}{\mu^{2}}{\vert R_{g_{B}}\vert}^{2}+\frac{1}{6}d\frac{\gamma}{\mu}-\frac{1}{3}cd(d+2)\\
+&\frac{1}{4}d(d+1)(d+2)c-\frac{1}{4}d^{2}c^{2}.\\
\end{aligned}
\end{equation}

Since $a_{2}(z,w)$ is a constant,  we get that $a_{2}(0,w)$ is a
constant, and thus $c_{0}=c_{1}=0$. Hence, from $2c_{0}+c_{1}=0$, we
have
\begin{equation*}
\frac{2}{3}dc(d-1)+\frac{1}{6}dc=\frac{2}{3}d^{2}c-\frac{1}{4}dc(d+1)(d+2),
\end{equation*}
in which we use the fact $\frac{\gamma}{\mu}=(d+1)-c$ (see
\eqref{300}). If $c\neq 0$, then we have
$$\frac{2}{3}(d-1)+\frac{1}{6}=
\frac{2}{3}d-\frac{1}{4}(d+1)(d+2).$$ Thus $d=0$ or $d=-3$, which is
impossible. Therefore,  we have $c=0,$ and furthermore, from
\eqref{300}, we have
\begin{equation}\label{39}
\mu=\frac{\gamma}{d+1}. \end{equation} By putting  $c=0$ and
$\;\frac{\gamma}{\mu}=d+1$ into $c_0=0$, we get
\begin{equation*}
\frac{1}{24}(d+1)^{2}{\vert
R_{g_{B}}\vert}^{2}-\frac{1}{12}d(d+1)=0.
\end{equation*}
That is
\begin{equation}\label{40}
[{\vert R_{g_{B}}\vert}^{2}]_{z=0}= \frac{2d}{d+1}.
\end{equation}
This proves the proposition.
\end{proof}

Now we will use  \eqref{equa17} to determine the Cartan-Hartogs
domain $(\Omega^{B}(\mu),g(\mu))$.

\noindent{\bf{Case 1.}} For
$\Omega=D_{m,n}^{\uppercase\expandafter{\romannumeral1}}:=\{z\in
M_{m\times n}: I-z\bar{z}^{t}>0\}$ $(1\leq m \leq n)$, we have
$${[{\vert
R_{g_{B}}\vert}^{2}]}_{z=0}=\frac{2mn(mn+1)}{(m+n)^{2}}.$$ By
(\ref{40}) (note $d=mn$ and $\gamma=n+m$  in this case), we get
$(mn+1)\frac{2mn(mn+1)}{(m+n)^{2}}=2mn,$ which implies $ m=1$ or
$n=1.$ So we get $m=1$. Then $\gamma=n+1$,  and  by (\ref{39}),
$\mu=1$. Hence the Cartan-Hartogs domain is the complex hyperbolic
space.

\noindent{\bf{Case 2.}} For
$\Omega=D_{n}^{\uppercase\expandafter{\romannumeral2}}:=\{z\in
M_{n,n}: \;z^{t}=-z,I-z\bar{z}^{t}>0\}\;\;(n\geq 4),$  we have
$${[{\vert R_{g_{B}}\vert}^{2}]}_{z=0}=\frac{n(n+1)(n^{2}-5n+12)-16n}{4(n-1)^{2}}.$$
By (\ref{40}) (note $d=n(n-1)/2$ in this case), we have $
n^5-5n^4+5n^3+5n^2-6n=0,$ which has no positive integer
solution for $n\geq 4$ (note that
$n^5-5n^4+5n^3+5n^2-6n=n(n-1)(n-2)(n+1)(n-3)$ has
no positive integer zero for $n\geq 4$).

\noindent{\bf{Case 3.}} For
$\Omega=D_{n}^{\uppercase\expandafter{\romannumeral3}}:=\{z\in
M_{n,n}: \; z^{t}=z, I-z\bar{z}^{t}>0\}(n\geq 2),$ we have
$${[{\vert R_{g_{B}}\vert}^{2}]}_{z=0}=\frac{n(n+1)(n^2+19n-60)+96n}{4(n+1)^{2}}.$$
By (\ref{40}) (note $d=n(n+1)/2$ in this case), we have
$n^{5}+21n^4-27n^3+11n^2-70n+64=0,$ which has no positive integer
solution for $n\geq 2$ (note that
$n^{5}+21n^4-27n^3+11n^2-70n+64=(n-1)(n^4+22n^3-5n^2+6n-64)$ has
no positive integer zero for $n\geq 2$).

\noindent{\bf{Case 4.}} For
$\Omega=D_{n}^{\uppercase\expandafter{\romannumeral4}}:=\{z\in\mathbb{C}^{n}:1-2z\bar{z}^{t}+{\vert
zz^{t}\vert}^{2}>0 ,z\bar{z}^{t}<1\}\; (n\geq5),$ we have
$${[\vert
R_{g_{B}}\vert}^{2}]_{z=0}=\frac{3n-2}{n}.$$ By (\ref{40}), we get
$n^{2}+n-2=0,$ which has no positive integer solution for $n\geq 5$.

\noindent{\bf{Cases 5 and 6.}} For an irreducible bounded symmetric
domain $\Omega$, we have that ${\vert
[R_{{g_{B}}}]_{\alpha\overline{\beta}\upsilon\overline{\delta}}(0)\vert}^{2}$
is an integer with respect to $(\Omega, g_{B})$ and
$${[{\vert R_{g_{B}}\vert}^{2}]}_{z=0}=\frac{1}{\gamma^{4}}\sum\limits_{\alpha,\beta,\upsilon,\delta}
{\vert
[R_{g_{B}}]_{\alpha\overline{\beta}\upsilon\overline{\delta}}(0)\vert}^{2}.$$

For $\Omega=D^{\uppercase\expandafter{\romannumeral5}}(16)=E_{6}/
Spin(10) \times T^{1}$ (in this case, $d=16$ and $\gamma=12$), by
(\ref{40}), we have
$${[{\vert R_{g_{B}}\vert}^{2}]}_{z=0}=\frac{32}{17}.$$
So
$\frac{32}{17}\gamma^{4}=\sum\limits_{\alpha,\beta,\upsilon,\delta}{\vert
[R_{g_{B}}]_{\alpha\overline{\beta}\upsilon\overline{\delta}}(0)\vert}^{2}$
is an integer, which is impossible for $\gamma=12$.

For $\Omega=D^{\uppercase\expandafter{\romannumeral6}}(27)=E_{7}/
E_{6} \times T^{1}$ (in this case, $d=27$ and $\gamma=18$), by
(\ref{40}), we have
$${[{\vert
R_{g_{B}}\vert}^{2}]}_{z=0}=\frac{27}{14}.$$ So
$\frac{27}{14}\gamma^{4}=\sum\limits_{\alpha,\beta,\upsilon,\delta}{\vert
[R_{g_{B}}]_{\alpha\overline{\beta}\upsilon\overline{\delta}}(0)\vert}^{2}$
is an integer, which is impossible for $\gamma=18$.

Combing the above results, we get that if $a_{2}$ is a constant,
then the Cartan-Hartogs domain is the complex hyperbolic space.

Since the complex hyperbolic space is the unit ball equipped with
the hyperbolic metric, we have that $a_{2}(z,w)$ is a constant for
the complex hyperbolic space. So we have proved the main theorem.

\section{Appendix}

For completeness, we will give ${[{\vert
R_{g_{B}}\vert}^{2}]}_{z=0}$ for a classical symmetric domain
$\Omega$ with the Bergman metric $g_{B}$ and prove  ${\vert
[R_{{g_{B}}}]_{\alpha\overline{\beta}\upsilon\overline{\delta}}(0)\vert}^{2}$
is an integer with respect to $(\Omega, g_{B})$
 for an irreducible bounded symmetric domain $\Omega$.  In fact, they can be found in some standard literatures (e.g.,
Helgason \cite{hel} and Mok \cite{Mok}).

{\bf I}. Here, we will give ${[{\vert R_{g_{B}}\vert}^{2}]}_{z=0}$
for a classical symmetric domain $\Omega$ with the Bergman metric
$g_{B}.$ By definition (see (13) in \cite{Zedda} for reference), we
have
$${{\vert
R_{g_{B}}\vert}^{2}}=\sum\limits_{\alpha,\beta,\eta,\theta,\zeta,\nu,\xi,\tau}\overline{g_B^{\alpha\bar{\zeta}}}
g_B^{\beta\bar{\nu}}\overline{g_B^{\eta\bar{\xi}}}g_B^{\theta\bar{\tau}}R_{\alpha\bar{\beta}\eta\bar{\theta}}\overline{R_{\zeta
\bar{\nu}\xi\bar{\tau}}}.$$ The curvature tensor $R_{g_{B}}$ of
$(\Omega, g_B)$ at $0$ can be found in section 2 in Calabi
\cite{Calabi}.

\noindent{\bf{Case 1.}} for
$\Omega=D_{m,n}^{\uppercase\expandafter{\romannumeral1}}:=\{z\in
M_{m\times n}:I-z\bar{z}^{t}>0\}$ (here $\gamma=m+n$, $d=mn$). Furthermore, we can give the following identity
\begin{equation*}
\begin{aligned}
\log K(z,z)=&\log\frac{1}{V(\Omega)}\textrm{det}(I-z\bar{z}^{t})^{-(n+m)}\\
=&\log{\frac{1}{V(\Omega)}}+(m+n)\sum \limits_{\alpha,\beta}{\vert z_{\alpha\beta}\vert}^{2}+\frac{m+n}{2}\sum\limits_{\alpha,\beta,\upsilon,\lambda} \bar{z}_{\alpha\upsilon}z_{\alpha\lambda}\bar{z}_{\beta\lambda}z_{\beta\upsilon}\\
&+\textrm{higher order terms}.
\end{aligned}
\end{equation*}
 Therefore we get $[g_{B}]_{\alpha\beta,\overline{\lambda\sigma}}(0)=(m+n)\delta^{\alpha}_{\lambda}\delta^{\beta}_{\sigma}$. Moreover, we have $
[R_{g_{B}}]_{\alpha\upsilon,\overline{\beta\rho},\lambda\sigma,\overline{\mu\tau}}(0)=-(m+n)(\delta^{\alpha}_{\beta}
\delta^{\lambda}_{\mu}\delta^{\upsilon}_{\tau}\delta^{\rho}_{\sigma}+\delta^{\alpha}_{\mu}
\delta^{\beta}_{\lambda}\delta^{\upsilon}_{\rho}\delta^{\sigma}_{\tau}).
$
Hence the following identity is established
\begin{equation*}
\begin{aligned}
{[{\vert
R_{g_{B}}\vert}^{2}]}_{z=0}=&\sum\limits_{\alpha,\beta,\lambda,\mu=1}^{m}\sum\limits_{\upsilon,\rho,\sigma,\tau=1}^{n}
\frac{1}{(m+n)^{2}}(\delta^{\alpha}_{\beta}
\delta^{\lambda}_{\mu}\delta^{\upsilon}_{\tau}\delta^{\rho}_{\sigma}+\delta^{\alpha}_{\mu}
\delta^{\beta}_{\lambda}\delta^{\upsilon}_{\rho}\delta^{\sigma}_{\tau})^{2}.\\
=&\frac{2mn(mn+1)}{(m+n)^{2}}
\end{aligned}
\end{equation*}

\noindent{\bf{Case 2.}} For
$\Omega=D_{n}^{\uppercase\expandafter{\romannumeral2}}:=\{z\in
M_{n,n},z^{t}=-z,I-z\bar{z}^{t}>0\}$ $(n\geq 4)$ $(\textrm{here }\gamma=2(n-1),d=\frac{1}{2}n(n-1))$. Similar to case 1, we have
\begin{equation*}
\begin{aligned}
\log K(z,z)=&\log \frac{1}{V(\Omega)}\textrm{det}(I-z\bar{z}^{t})^{-(n-1)}\\
=&-\log{V(\Omega)}+(n-1)\sum \limits_{\alpha<\beta}2{\vert z_{\alpha\beta}\vert}^{2}+\frac{n-1}{2}\sum\limits_{\alpha,\beta,\upsilon,\lambda} \bar{z}_{\alpha\upsilon}z_{\alpha\lambda}\bar{z}_{\beta\lambda}z_{\beta\upsilon}\\
&+\textrm{higher order terms}.
\end{aligned}
\end{equation*}
Hence we get $[g_{B}]_{\alpha\beta,\overline{\lambda\sigma}}(0)=2(n-1)\delta^{\alpha}_{\lambda}\delta^{\beta}
_{\sigma}(\alpha<\beta,\;\lambda<\sigma)$ and similar to (3.9) in Calabi
\cite{Calabi},
we have $$
[R_{g_{B}}]_{\alpha\upsilon,\overline{\beta\rho},\lambda\sigma,\overline{\mu\tau}}(0)=2(n-1)(-\delta^{\beta\rho}_
{\alpha\sigma}\delta^{\mu\tau}_
{\lambda\upsilon}-\delta^{\mu\tau}_
{\beta\sigma}\delta^{\beta\rho}_
{\lambda\upsilon}+\delta^{\beta\rho}_
{\alpha\lambda}\delta^{\mu\tau}_
{\sigma\upsilon}+\delta^{\mu\tau}_
{\alpha\lambda}\delta^{\beta\rho}_
{\sigma\upsilon}). $$
Where the precisely definition of $\delta^{\alpha\beta}_
{\rho\sigma}(=\frac{\partial z_{\alpha\beta}}{\partial z_{\rho\sigma}}=\delta^{\alpha}_{\rho}\delta^{\beta}_{\sigma}-\delta^{\alpha}_{\sigma}\delta^{\beta}_{\rho})$ can be found in Calabi
\cite{Calabi}. Here we must note that $z_{\alpha\beta}=-z_{\beta\alpha}$ and $\rho<\sigma$.
Hence after a long computation yields the following result
\begin{equation*}
\begin{aligned}
{[{\vert R_{g_{B}}\vert}^{2}]}_{z=0}=&\sum\limits_{\alpha<\upsilon\;\beta<\rho\;\lambda<\sigma\;\mu<\tau}
g_{B}^{\overline{\alpha\upsilon,\overline{\alpha\upsilon}}}g_{B}^{\beta\rho,\overline{\beta\rho}}
g_{B}^{\overline{\lambda\sigma,\overline{\lambda\sigma}}}g_{B}^{\mu\tau,\overline{\mu\tau}}[R_{g_{B}}]
_{\alpha\upsilon,\overline{\beta\rho},\lambda\sigma,\overline{\mu\tau}}\overline{[R_{g_{B}}]
_{\alpha\upsilon,\overline{\beta\rho},\lambda\sigma,\overline{\mu\tau}}}\\
=&\frac{n(n+1)(n^{2}-5n+12)-16n}{4(n-1)^{2}}.\\
\end{aligned}
\end{equation*}
Therefore, we have
\begin{equation*}
\begin{aligned}
{[{\vert R_{g_{B}}\vert}^{2}]}_{z=0}=\frac{n(n+1)(n^{2}-5n+12)-16n}{4(n-1)^{2}}&.\\
\end{aligned}
\end{equation*}

\noindent{\bf{Case 3.}} For
$\Omega=D_{n}^{\uppercase\expandafter{\romannumeral3}}:=\{z\in
M_{n,n},z^{t}=z,I-z\bar{z}^{t}>0\}\;(n\geq 2)$ $(\textrm{here } \gamma=n+1, d=\frac{1}{2}n(n+1))$. Similarly, the $\log K(z,z)$ is given by
 \begin{equation*}
 \begin{aligned}
\log K(z,z)=&\log \frac{1}{V(\Omega)}\textrm{det}(I-z\bar{z}^{t})^{-(n+1)}\\
=&-\log{V(\Omega)}+(n+1)\sum \limits_{\alpha<\beta}2{\vert z_{\alpha\beta}\vert}^{2}+(n+1)\sum \limits_{\alpha=\beta}{\vert z_{\alpha\beta}\vert}^{2}\\
&+\frac{n+1}{2}\sum\limits_{\alpha,\beta,\lambda,\upsilon} \bar{z}_{\alpha\upsilon}z_{\alpha\lambda}\bar{z}_{\beta\lambda}z_{\beta\upsilon}+\textrm{higher order terms}.\\
\end{aligned}
\end{equation*}
Hence we have
\begin{equation*}
[g_{B}]_{\alpha\beta,\overline{\lambda\sigma}}(0)=
\begin{cases}
2(n+1)\delta^{\alpha}_{\lambda}\delta^{\beta}_{\sigma},&\alpha<\beta,\;\lambda<\sigma\\
(n+1)\delta^{\alpha}_{\lambda},&\alpha=\beta,\;\lambda=\sigma\\
\end{cases}
\end{equation*}

Similar to (3.12) in Calabi
\cite{Calabi}, we have $[R_{g_{B}}]_{\alpha\upsilon,\overline{\beta\rho},\lambda\sigma,\overline{\mu\tau}}(0)$ equals
\begin{equation*}
\begin{cases}-2(n+1)(e^{\beta\rho}_
{\alpha\sigma}e^{\mu\tau}_
{\lambda\upsilon}+e^{\beta\rho}_
{\lambda\upsilon}e^{\mu\tau}_
{\alpha\sigma}+e^{\beta\rho}_
{\alpha\lambda}e^{\mu\tau}_
{\sigma\upsilon}+e^{\beta\rho}_
{\sigma\upsilon}e^{\mu\tau}_
{\alpha\lambda}),&\alpha<\upsilon, \beta<\rho, \lambda<\sigma, \mu<\tau\\
-2(n+1)(e^{\alpha\upsilon}_
{\beta\tau}e^{\lambda\sigma}_
{\mu\beta}+e^{\alpha\upsilon}_
{\mu\beta}e^{\lambda\sigma}_
{\beta\tau}),&\alpha<\upsilon,\beta=\rho,\lambda<\sigma,\mu<\tau\\
-2(n+1)\delta^{\beta}_{\alpha}(e^{\mu\tau}_
{\alpha\lambda}+e^{\mu\tau}_
{\sigma\alpha}),&\alpha=\upsilon,\beta=\rho,\lambda<\sigma,\mu<\tau\\
-(n+1)(e^{\alpha\upsilon}_
{\beta\mu}e^{\lambda\sigma}_
{\mu\beta}+e^{\alpha\upsilon}_
{\mu\beta}e^{\lambda\sigma}_
{\beta\mu}),&\alpha<\upsilon,\beta=\rho,\lambda<\sigma,\mu=\tau\\
0,&\alpha<\upsilon, \beta=\rho,\mu=\tau,\lambda=\sigma\\
-2(n+1),&\alpha=\upsilon,\beta=\rho,\mu=\tau,\lambda=\sigma\\
\end{cases}
\end{equation*}
Where the exact description of $e^{\alpha\beta}_
{\rho\sigma}(=\frac{\partial z_{\alpha\beta}}{\partial z_{\rho\sigma}})$ can also be consulted in Calabi
\cite{Calabi}. Here $z_{\alpha\beta}=z_{\beta\alpha}$ and $\rho\leq\sigma$.
Hence after a complicated computation, we have
\begin{align*}
{[{\vert
R_{g_{B}}\vert}^{2}]}_{z=0}=&\sum\limits_{\alpha\leq\upsilon\;\beta\leq\rho\;\lambda\leq\sigma\;\mu\leq\tau}
g_{B}^{\overline{\alpha\upsilon,\overline{\alpha\upsilon}}}g_{B}^{\beta\rho,\overline{\beta\rho}}
g_{B}^{\overline{\lambda\sigma,\overline{\lambda\sigma}}}g_{B}^{\mu\tau,\overline{\mu\tau}}[R_{g_{B}}]
_{\alpha\upsilon,\overline{\beta\rho},\lambda\sigma,\overline{\mu\tau}}\overline{[R_{g_{B}}]
_{\alpha\upsilon,\overline{\beta\rho},\lambda\sigma,\overline{\mu\tau}}}\\
=&\frac{n(n+1)(n^2+19n-60)+96n}{4(n+1)^{2}}.\\
\end{align*}
\noindent{\bf{Case 4.}} For
$\Omega=D_{n}^{\uppercase\expandafter{\romannumeral4}}:=\{z\in\mathbb{C}^{n}:1-2z\bar{z}^{t}+{\vert
zz^{t}\vert}^{2}>0 ,z\bar{z}^{t}<1\} (n\geq5)$ $(\textrm{here } \gamma=n, d=n)$. Moreover, $\log K(z,z)$ can be expressed by
\begin{align*}
\log K(z,z)=&\log\frac{1}{V(\Omega)}(1-2z\bar{z}^{t}+{\vert zz^{t}\vert}^{2})^{-n}\\
=&-\log{V(\Omega)}+2n\sum\limits_{i}{\vert z_{i}\vert}^{2}-n{\vert \sum\limits_{i}{z_{i}^{2}}\vert}^{2}+2n(\sum\limits_{i}{\vert z_{i}\vert}^{2})^{2}+\textrm{higher order terms}.
\end{align*}
Hence we have $[g_{B}]_{\alpha\beta}(0)=2n\delta^{\alpha}_{\beta}$ and
$[R_{g_{B}}]_{\alpha\bar{\rho}\beta\bar{\sigma}}(0)=-4n(\delta^{\alpha}_{\rho}\delta^{\beta}_{\sigma}+\delta^{\alpha}_{
\sigma}
\delta^{\beta}_{\rho}-\delta^{\alpha}_{\beta}\delta^{\rho}_{\sigma}).$ Hence
$$
{[{\vert
R_{g_{B}}\vert}^{2}]}_{z=0}=\frac{1}{16n^4}\sum\limits_{\alpha,\beta,\rho,\sigma}
16n^2(\delta^{\alpha}_{\rho}\delta^{\beta}_{\sigma}+\delta^{\alpha}_{
\sigma}
\delta^{\beta}_{\rho}-\delta^{\alpha}_{\beta}\delta^{\rho}_{\sigma})^{2}
=\frac{3n-2}{n}.$$

{\bf II}. Here, for an irreducible bounded symmetric domain
$\Omega$, we will prove that ${\vert
[R_{{g_{B}}}]_{\alpha\overline{\beta}\upsilon\overline{\delta}}(0)\vert}^{2}$
is an integer with respect to $(\Omega, g_{B})$ and
$${[{\vert R_{g_{B}}\vert}^{2}]}_{z=0}=\frac{1}{\gamma^{4}}\sum\limits_{\alpha,\beta,\upsilon,\delta}
{\vert
[R_{g_{B}}]_{\alpha\overline{\beta}\upsilon\overline{\delta}}(0)\vert}^{2}.$$

 Firstly, we will express the curvature tensor in terms of Lie brackets of root vectors. All the following conventions can be found in Siu \cite{Siu2}, Borel \cite{borel} and the book Helgason \cite{hel}. So we will not explain it explicitly. And we will just compute the curvature tensor of the irreducible compact Hermitian symmetric manifold $G/K$ which will not affect our conclusions.

  Let $\Psi$ denote the set of nonzero roots of ${\mathfrak{g}}_{c}$ with respect to ${\mathfrak{t}}_{c}$. Write $\Delta$ to denote that of nonzero noncompact roots and $\Delta^{+}$ that of all positive noncompact roots. Moreover, there exists a set $\Lambda$ of strongly orthogonal noncompact positive roots.  For every $\alpha\in \Delta^{+}$, let $e_{\alpha}$ denote the root vector for the root $\alpha$, $e_{-a}=\overline{e_{\alpha}}$ denotes the root vectors for the root $-\alpha$. Then we have the (direct sum) root space decomposition $${\mathfrak{g}}_{c}={\mathfrak{t}}_{c}+\sum\limits_{\alpha\in \Psi}\mathbb{C}e_{\alpha}.$$ This decomposition is orthogonal with respect to the Killing form $B(\cdot,\bar{\cdot})$. Since the Killing form on $\mathfrak{g}$ is negative definite, then, we can modify \cite{hel}(P.176), Thm. 5.5 by the following results.
\begin{theorem1}[see also \cite{Van} Lemma. 4.3.22 and Thm. 4.3.26]\label{Thm3.1}
For each $\alpha\in \Delta^{+}$, let $X_{\alpha}$ be any root
vector, then we have
\begin{equation*}
{[X_{\alpha},X_{-\beta}]}=
\begin{cases}
N_{\alpha,-\beta}X_{\alpha-\beta},& \alpha-\beta\in\Psi,\; \alpha\neq\beta\\
0,& \alpha-\beta \not\in\Psi,\;\alpha\neq\beta\\
B(X_{\alpha},X_{-\alpha})H_{\alpha} \in \mathfrak{t}_{c},& \alpha=\beta\\
\end{cases}
\end{equation*}
\begin{equation*}
\begin{aligned}
N_{\alpha,-\beta}^{2}=&-\frac{q(1-p)}{2}(\alpha,\alpha)B(X_{\alpha},X_{-\alpha}),\\
\end{aligned}
\end{equation*}
where $n\alpha-\beta (p \leq n\leq q)$ is the $\alpha-$series containing $\beta$ and $(\alpha,\alpha)=B(H_{\alpha},H_{\alpha}).$
\end{theorem1}

Let $\mathfrak{p}_{+}=\bigoplus\limits_{\alpha\in
\Delta^{+}}\mathbb{C}e_{\alpha}$ and
$\mathfrak{p}_{-}=\bigoplus\limits_{-\alpha\in
\Delta^{+}}\mathbb{C}e_{\alpha}$. Then from \cite{Siu2},
we have $$T_{p}^{1,0}\Omega=\mathfrak{p}_{+},\;T_{p}^{0,1}\Omega=\mathfrak{p}_{-}.$$ Moreover, $-B(\cdot,\overline{\cdot})$ induces an
invariant metric on $\Omega$. Thus the Hermitian metric
$\langle\cdot,\cdot\rangle$ on $T_{p}^{1,0}\Omega$ is defined by
$$\langle e_{\alpha},e_{\beta}\rangle=\langle
e_{\alpha},\overline{e_{\beta}}\rangle_{R}=-B(e_{\alpha},e_{-\alpha})\;\;(\alpha,\beta
\in \Delta^{+}).$$ The curvature tensor $R$ is given by
$R(X,Y)Z=-[[X,Y],Z].$ The paper \cite{Siu2} also tells us that
$R_{\alpha\overline{\beta}\upsilon\overline{\delta}}$ with respect
to the $\langle \cdot,\cdot\rangle$ can be expressed by
\begin{equation}\label{R}
R_{\alpha\overline{\beta}\upsilon\overline{\delta}}=-\langle[e_{\alpha},e_{-\beta}],[e_{\delta},e_{-\upsilon}]\rangle.
\end{equation}

 It is well known that the Bergman metric is an invariant metric. Hence by \cite{Mok}(Chapter 3, 2.1), we have $g_{B}=a\langle\cdot,\cdot\rangle$ where $a$ is a positive constant. By the Proposition 2 in \cite{Ko} , we know that $[R_{g_{B}}]_{\alpha\overline{\alpha}\alpha\overline{\alpha}}(0)=-2\gamma(\alpha\in\Lambda)$ and $[g_{B}]_{\alpha\overline{\beta}}(0)=\gamma\delta_{\alpha\beta}$. Thus by the definition, we have $${[{\vert R_{g_{B}}\vert}^{2}]}_{z=0}=\frac{1}{\gamma^{4}}\sum\limits_{\alpha,\beta,\upsilon,\delta}
{\vert
[R_{g_{B}}]_{\alpha\overline{\beta}\upsilon\overline{\delta}}(0)\vert}^{2
}.$$ Without loss of generality, we can assume that $\{e_{\alpha}\}$ constitutes the corresponding basis. Hence we have
 \begin{equation}\label{equagB}
 [g_{B}]_{\alpha\overline{\beta}}(0)=a\langle e_{\alpha},e_{\beta}\rangle=\gamma\delta_{\alpha\beta}.
 \end{equation}
 Thus we get ${\vert [R_{{g_{B}}}]_{\alpha\overline{\beta}\upsilon\overline{\delta}}\vert}^{2}(0)=a^{2}{\vert R_{\alpha
 \overline{\beta}\upsilon\overline{\delta}}(0)\vert}^{2}$. Now combined with \cite{borel} Lemma 2.1, we have the following result
\begin{theorem1}\label{Thm3.2}
For an irreducible bounded symmetric domain $\Omega$, we have ${\vert
[R_{{g_{B}}}]_{\alpha\overline{\beta}\upsilon\overline{\delta}}(0)\vert}^{2}$
is an integer with respect to $(\Omega, g_{B})$.
\end{theorem1}
\begin{proof}
Firstly, by the \cite{borel} Lemma 2.1, Thm \ref{Thm3.1} and (\ref{R}), it is not hard to get that
for any $\alpha\;,\beta\;,\upsilon\;,\delta\in \Delta^{+}$
\begin{equation}\label{equ21}
{\vert
[R_{{g_{B}}}]_{\alpha\overline{\beta}\upsilon\overline{\delta}}\vert}^{2}(0)=
\begin{cases}
\gamma^{2}N_{\alpha,-\beta}^{2}N_{\delta,-\upsilon}^{2}&
\alpha-\beta=\delta-\upsilon,\; \delta\neq\upsilon\\
0&\alpha-\beta\neq\delta-\upsilon\\
\gamma^{2}N_{\alpha,-\upsilon}^{4}& \alpha=\beta,\upsilon=\delta,\alpha\neq\upsilon\\
a^{2}B(e_{\alpha},e_{-\alpha})^{4}\langle H_{\alpha},H_{\alpha}\rangle^{2}&\alpha=\beta=\upsilon=\delta.\\
\end{cases}
\end{equation}

For the classical irreducible bounded symmetric domains, we have ${\vert
[R_{{g_{B}}}]_{\alpha\overline{\beta}\upsilon\overline{\delta}}(0)\vert}^{2}$
is an integer with respect to $(\Omega, g_{B})$ by the above arguments. For the two exceptional bounded symmetric domains, by Helgason
\cite{hel}(p523. 7), we know that
$(\alpha,\alpha)=B(H_{\alpha},H_{\alpha})=\frac{1}{\gamma}$
for all $\alpha\in \Delta^{+}$. What's more, by $(\ref{equagB})$, we
know that $B(e_{\alpha},e_{-\alpha})=-\frac{\gamma}{a}$. Hence
combined with \cite{Ko} Proposition 2 and (\ref{equ21}), for $\alpha\in\Lambda$ we
have
\begin{equation}\label{eauRB}
{\vert
[R_{{g_{B}}}]_{\alpha\overline{\beta}\upsilon\overline{\delta}}\vert}^{2}(0)=a^2\frac{\gamma^4}{a^4}B(H_{\alpha},
\overline{H_{\alpha}})^{2}=4\gamma^2.
\end{equation}
Since
$B(H_{\alpha},\overline{H_{\alpha}})=-B(H_{\alpha},H_{\alpha})=-\frac{1}{\gamma}$.
Hence we have $a=\frac{1}{2}$ and
${\vert
[R_{{g_{B}}}]_{\alpha\overline{\beta}\upsilon\overline{\delta}}\vert}^{2}(0)=4\gamma^2$
for $\alpha\in\Delta^{+}$.

 Furthermore, by Thm. \ref{Thm3.1}, we know that $$N_{\alpha,-\beta}^{2}=-\frac{q_{1}(1-p_{1})}{2}\alpha(H_{\alpha})B(e_{\alpha},e_{-\alpha})=\frac{q_{1}(1-p_{1})}
 {2\gamma}
 2\gamma={q_{1}(1-p_{1})}$$ Hence we have $N_{\alpha,-\beta}^{2}$ is an integer. Similarly, we have $N_{\delta,-\upsilon}^{2}$ and $N_{\alpha,-\upsilon}^{2}$ are integers. Then, combined with (\ref{equ21}) and (\ref{eauRB}), it is easy to show that ${\vert [R_{{g_{B}}}]_{\alpha\overline{\beta}\upsilon\overline{\delta}}\vert}^{2}(0)$ is an integer for the two exceptional bounded symmetric domains.
So far we complete the proof.
\end{proof}

\vskip 10pt

\noindent\textbf{Acknowledgments}\quad We sincerely thank the referees, who read the paper very carefully and gave many useful suggestions. The second author was supported by
the National Natural Science Foundation of China (No.11671306).


\begin{thebibliography}{99}
\bibliographystyle{plain}

\bibitem{borel}
Borel, A.: On the curvature tensor of the Hermitian symmetric
manifolds. Ann. Math. \textbf{71}(3), 508-521(1960)

\bibitem{CGR}Cahen, M., Gutt, S., Rawnsley, J.: Quantization of K\"{a}hler manifolds. I: Geometric interpretation of Berezin's quantization.
J. Geom. Phys. \textbf{7}, 45-62 (1990)
\bibitem{Calabi}
Calabi, E., Vesentini, E: On compact, locally symmetric Kahler manifolds. Ann. Math.\textbf{71}(3), 472-507(1960)
\bibitem{Cat}Catlin, D.: The Bergman kernel and a theorem of Tian.
Analysis and geometry in several complex variables (Katata, 1997),
Trends Math., Birkh\"{a}user Boston, Boston, MA, 1-23 (1999)

\bibitem{Donaldson}Donaldson, S.: Scalar curvature and projective
embeddings, I.  J. Diff. Geom. \textbf{59}, 479-522 (2001)

\bibitem{End2}
Engli\v{s}, M.: A Forelli-Rudin construction and asymptotics of
weighted Bergman kernels. J. Funct. Anal. \textbf{177}(2),
257-281(2000)
\bibitem{Eng1}
Engli\v{s}, M.: The asymptotics of a Laplace integral on a Kahler
manifold. J. Reine Angew. Math. \textbf{528}, 1-39(2000)
\bibitem{F-Tu}
Feng, Z. and Tu, Z.: On canonical metrics on Cartan-Hartogs domains.
Math. Z. \textbf{278}(1), 301-320(2014)
\bibitem{hel}
Helgason, S.: Differential geometry, Lie groups, and symmetric spaces. Academic press, 1979.
\bibitem{Ko}
Kor\'{a}nyi, A.: Analytic invariants of bounded symmetric domains.
Proc. Amer. Math. Soc. \textbf{19}(2), 279-284(1968)

\bibitem{LZ}Loi, A. and Zedda, M.: Balanced metrics on Cartan and Cartan-Hartogs domains. Math. Z. \textbf{270}, 1077-1087 (2012)


\bibitem{Lu}Lu, Z.: On the lower order terms of the asymptotic
expansion of Tian-Yau-Zelditch. Amer. J. Math. \textbf{122}(2),
235-273 (2000)

\bibitem{MM07}Ma, X. and  Marinescu, G.: Holomorphic Morse
inequalities and Bergman kernels. Progress in Mathematics, Vol. 254,
Birkh\v{a}user Boston Inc., Boston, MA (2007)

\bibitem{MM08}Ma, X.
and  Marinescu, G.: Generalized Bergman kernels on symplectic
manifolds. Adv. Math. \textbf{217}(4), 1756-1815 (2008)

\bibitem{MM12}Ma, X. and Marinescu, G.: Berezin-Toeplitz
quantization on K\"{a}hler manifolds. J. reine angew. Math.
\textbf{662}, 1-56 (2012)


\bibitem{Mok}
Mok, N.: Metric rigidity theorems on Hermitian locally symmetric manifolds. World Scientific(1989)

\bibitem{R} Rawnsley, J.: Coherent states and K\"{a}hler manifolds.  Quart. J. Math. Oxford \textbf{28}(2), 403-415 (1977)

\bibitem{Siu2}
Siu, Y.-T.: Strong rigidity of compact quotients of exceptional
bounded symmetric domains. Duke Math. J. \textbf{48}(4),
857-871(1981)
\bibitem{Van}
V, S, Varadarajan.: Lie Groups, Lie Algebras, and Their Representation. Spring(1984)

\bibitem{X}Xu, H.: A closed formula for the asymptotic expansion of the Bergman kernel. Commun. Math. Phys. \textbf{314}, 555-585 (2012)

\bibitem{YW}Yin, W.P. and
 Wang,  A.: The equivalence on classical metrics. Sci. China  Ser. A: Mathematics \textbf{50}(2), 183-200 (2007)

\bibitem{Zedda}
Zedda, M.: Canonical metrics on Cartan-Hartogs domains. International Journal of Geometric Methods in Modern Physics. \textbf{9}(01)(2012)

\bibitem{Zeld}Zelditch, S.: Szeg\"{o} kernels and a theorem of
Tian. Internat. Math. Res. Notices, No. 6, 317-331 (1998)

\end{thebibliography}
\end{document}